\title{Poisson approximation for non-backtracking random walks}
\author{{Noga Alon\thanks{Schools of Mathematics and Computer Science,
Raymond and Beverly Sackler Faculty of Exact Sciences, Tel Aviv
University, Tel Aviv, 69978, Israel. Email: nogaa@tau.ac.il.
Research supported in part by a USA-Israeli BSF grant, by the Israel
Science Foundation and by the Hermann Minkowski Minerva Center for
Geometry at Tel Aviv University.}} \quad {Eyal Lubetzky
\thanks{ School of Mathematics, Raymond and Beverly
Sackler Faculty of Exact Sciences, Tel Aviv University, Tel Aviv,
69978, Israel. Email: lubetzky@tau.ac.il. Research partially
supported by a Charles Clore Foundation Fellowship.}}}

\documentclass [11pt]{article}
\usepackage[]{amsmath,amssymb,amsfonts,latexsym,amsthm,enumerate}
\usepackage[]{graphicx}
\date{}

\newtheorem{theorem}{Theorem}[section]

\newtheorem{claim}[theorem]{Claim}

\newtheorem{corollary}[theorem]{Corollary}
\newtheorem{proposition}[theorem]{Proposition}

\renewcommand{\epsilon}{\varepsilon}

\newcommand{\var}{\operatorname{Var}}
\newcommand{\Po}{\operatorname{Po}}

\newtheoremstyle{upright}%
        {8pt plus2pt minus4pt}%
        {8pt plus2pt minus4pt}%
        {\upshape}%
        {}%
        {\bfseries}%
        {:}%
        {1em}%
        {}%
\theoremstyle{upright}
\newtheorem{remark}[theorem]{Remark}

\newcommand{\ignore}[1]{}

\begin{document}
\maketitle

\begin{abstract}
Random walks on expander graphs were thoroughly studied, with the
important motivation that, under some natural conditions, these
walks mix quickly and provide an efficient method of sampling the
vertices of a graph. The authors of \cite{ABLS} studied
non-backtracking random walks on regular graphs, and showed that
their mixing rate may be up to twice as fast as that of the simple
random walk. As an application, they showed that the maximal number
of visits to a vertex, made by a non-backtracking random walk of
length $n$ on a high-girth $n$-vertex regular expander, is typically
$(1+o(1))\frac{\log n}{\log\log n}$, as in the case of the balls and
bins experiment. They further asked whether one can establish the
precise distribution of the visits such a walk makes.

In this work, we answer the above question by combining a
generalized form of Brun's sieve with some extensions of the ideas
in \cite{ABLS}. Let $N_t$ denote the number of vertices visited
precisely $t$ times by a non-backtracking random walk of length $n$
on a regular $n$-vertex expander of fixed degree and girth $g$. We
prove that if $g=\omega(1)$, then for any fixed $t$, $N_t/n$ is
typically $\frac{1}{\mathrm{e}t!}+o(1)$. Furthermore, if
$g=\Omega(\log\log n)$, then $N_t/n$ is typically
$\frac{1+o(1)}{\mathrm{e}t!}$ uniformly on all $t \leq
(1-o(1))\frac{\log n}{\log\log n}$ and $0$ for all $t \geq
(1+o(1))\frac{\log n}{\log\log n}$. In particular, we obtain the
above result on the typical maximal number of visits to a single
vertex, with an improved threshold window. The essence of the proof
lies in showing that variables counting the number of visits to a
set of sufficiently distant vertices are asymptotically independent
Poisson variables.
\end{abstract}

\section{Introduction}
\subsection{Background and definitions}
Random walks on graphs have played a major role in Theoretical and
Applied Computer Science, as under some natural requirements
(related to the notion of expander graphs), these walks converge
quickly to a unique stationary distribution, and enable efficient
sampling of this distribution. This fact was exploited for example
in \cite{AKLLR}, \cite{BF} and \cite{Reingold}, in the study of
space efficient algorithms for $S-T$ connectivity in undirected
graphs. Another well known example is the conservation of random
bits in the amplification of randomized algorithms (we will
elaborate on this point later on). In many applications for random
walks, it seems that using non-backtracking random walks may yield
better results, and that these walks possess better random-looking
properties than those of simple random walks. This motivated the
authors of \cite{ABLS} to study the mixing-rate of a
non-backtracking random walk, and show that it may be up to twice as
fast as that of a simple random walk. They further show that
 the number of times that such a walk visits the vertices
of a high-girth expander is random-looking, in the sense that its
maximum is typically asymptotically the same as the maximal load of
the classical balls and bins experiment. In this paper, we further
examine this setting, and answer a question raised in \cite{ABLS} by
giving a precise description of the limiting distribution of these
visits.

We briefly mention some well known properties of random walks on
regular graphs; for further information, see, e.g., \cite{Lovasz},
\cite{Sinclair}. Let $G = (V,E)$ denote a $d$-regular undirected
graph on $n$ vertices. A random walk of length $k$ on $G$ from a
given vertex $w_0$ is a uniformly chosen member $W \in \mathcal{W}$,
where $\mathcal{W}=\{(w_0,w_1,\ldots,w_k) : w_{i-1}w_i \in E\}$ is
the set of all paths of length $k$ starting from $w_0$.
Alternatively, a random walk on $G$, $\mathcal{M}$, is a Markov
chain whose state space is $V$, where the transition probability
from $u$ to $v$ is $P_{u v} = \mathbf{1}_{u v \in E} / d$. The
transition probability matrix of $\mathcal{M}$ is doubly stochastic,
and the uniform distribution $\pi(u)=1/n$ is a stationary
distribution of $\mathcal{M}$. If $G$ is connected and non-bipartite
then $\mathcal{M}$ is irreducible and aperiodic, in which case it
converges to the unique stationary distribution $\pi$, regardless of
the starting point $w_0$. These two sufficient and necessary
conditions have a clear formulation in terms of the spectrum of $G$,
which also dictates the rate at which $\mathcal{M}$ converges to
$\pi$.

The adjacency matrix of $G$ is symmetric and thus has $n$ real
eigenvalues, all at most $d$ in absolute value (by the
Perron-Frobenius Theorem). Let $d=\lambda_1 \geq \lambda_2 \geq
\ldots \geq \lambda_n$ denote these eigenvalues. It is simple and
well known (see, e.g., \cite{AlgabraicGraphTheory}) that the
multiplicity of the eigenvalue $d$ is equal to the number of
connected components of $G$, and the minimal eigenvalue $\lambda_n$
is equal to $-d$ iff $G$ has a bipartite connected component.
Therefore, letting $ \lambda = \max_{i
> 1} |\lambda_i|$ denote the maximal absolute value of all non-trivial eigenvalues,
we obtain that $G$ is connected and non-bipartite iff $\lambda < d$.
The quantity $d-\lambda$ is often referred to as the {\em spectral
gap} of $G$, and is strongly related to the expansion properties of
the graph. In particular, when $d-\lambda$ is bounded from below by
a constant, we call the graph an {\em expander} (a closely related
notion of expander graphs is defined by the expansion ratio of each
set of at most $n/2$ vertices to its neighbor vertices in the
graph). See \cite{KS} for a survey on the many fascinating
pseudo-random properties exhibited by such graphs.

An $(n,d,\lambda)$ graph is a $d$-regular graph on $n$ vertices,
whose largest non-trivial eigenvalue in absolute value is $\lambda$.
As mentioned above, the condition $\lambda < d$ is sufficient and
necessary for the random walk on $G$ to converge to $\pi$. In this
case, the quantity $\lambda / d$ governs the rate of this
convergence: the {\em mixing rate} of $\mathcal{M}$ is defined to be
$\limsup_{k\to\infty}|P_{u v}^{(k)}-\pi(v) |^{1/k}$, where
$P^{(k)}_{u v}$ is the probability that $\mathcal{M}$ reaches $v$ in
the $k$-th step given that it started from $u$. It is well known
(see, for instance, \cite{Lovasz}) that the mixing rate of the
simple random walk on an $(n,d,\lambda)$ graph is $\lambda / d$, and
that in fact, the $L_2$ distance between $\pi$ and the distribution
of $\mathcal{M}$ after $k$ steps is at most $(\lambda / d)^k$.
Therefore, after $\Omega(\log_{d/\lambda} n)$ steps, the $L_2$
distance between the distribution of the simple random walk and the
uniform distribution is at most $1/n^{\Omega(1)}$.

A useful and well known application of random walks on expander
graphs is the following result, related to the conservation of
random bits (cf., e.g., \cite{ProbMethod}, Corollary 9.28). Suppose
$G$ is an $(n,d,\lambda)$ graph, and $U$ is a predefined set of
$\alpha n$ vertices, for some $\alpha > 0$. Then a random walk of
length $k$, starting from a random vertex, avoids $U$ with
probability at most $(1-\alpha + \alpha\frac{\lambda}{d})^k$.
Indeed, the term $(1-\alpha)^k$ is the probability to miss $U$ when
selecting $k$ vertices, uniformly and independently, in which case
we would require $k \log n$ random bits for the selection process.
Using a random walk, we require only $\log n + k \log d$ random
bits, at the cost of increasing the base of the exponent by an
additive term of $\alpha \frac{\lambda}{d}$. This enables amplifying
the error-probability in randomized algorithms (such as the
Rabin-Miller primality testing algorithm) using fewer random bits:
an algorithm utilizing $s$-bit seeds can be amplified $k$ times
using $s + \Theta(k)$ random bits via a walk on a constant-degree
expander, instead of $s k$ random bits in the naive approach.

In many applications of random walks on graphs, forbidding the
random walk to backtrack appears to produce better results; an
example of this is the construction of sparse universal graphs in
\cite{AC}, where a crucial element is a non-backtracking random walk
on a high-girth expander. A non-backtracking random walk of length
$k$ on $G$, starting from some vertex $w_0$, is a sequence
$\widetilde{W}=(w_0,\ldots,w_k)$, where $w_i$ is chosen uniformly
over all neighbors of $w_{i-1}$ excluding $w_{i-2}$. The mixing-rate
of a non-backtracking random walk on a regular graph, in terms of
its eigenvalues, was computed in \cite{ABLS}, using some properties
of Chebyshev polynomials of the second kind. It is shown in
\cite{ABLS} that this rate is always better than that of the simple
random walk, provided that $d=n^{o(1)}$. In fact, the mixing rate of
the non-backtracking random walk may be up to twice faster, and the
closer the graph is to being a Ramanujan graph (that is, a graph
satisfying $\lambda \leq 2\sqrt{d-1}$), the closer the ratio between
the two mixing-rates is to $2(d-1)/d$.

As an application, the authors of \cite{ABLS} analyzed the maximal
number of visits that a non-backtracking random walk of length $n$
makes to a vertex of $G$, an $(n,d,\lambda)$ graph of fixed degree
and girth $\Omega(\log\log n)$. Using a careful second moment
argument, they proved that this quantity is typically
$(1+o(1))\frac{\log n}{\log\log n}$, as is the typical maximal
number of balls in a single bin when throwing $n$ balls to $n$ bins
uniformly at random (more information on the classical balls and
bins experiment may be found in \cite{Feller}) . In contrast to
this, it is easy to see that a typical \textbf{simple} random walk
of length $n$ on a graph as above visits some vertex $\Omega(\log
n)$ times. The authors of \cite{ABLS} further asked whether it is
possible to establish the precise distribution of the number of
visits that a non-backtracking random walk makes on a graph $G$ as
above.

In this paper, we answer the above question, by combining a
generalized form of Brun's Sieve with extensions of some of the
ideas in \cite{ABLS}. This approach shows that even if the girth of
$G$ grows to infinity arbitrarily slowly with $n$, then for any
fixed $t$, the fraction of vertices visited precisely $t$ times is
typically $(1+o(1))n/(\mathrm{e}t!)$, where the $o(1)$-term tends to
$0$ as $n\to\infty$. The extension of Brun's Sieve, which includes
an estimate on the convergence rate of the variables, is used to
treat the case where $t$ depends on $n$, and allows the
characterization of the number of vertices visited $t$ times for all
$t$. In particular, this provides an alternative proof for the
result of \cite{ABLS} of the typical maximum number of visits to a
vertex, with an improved error term. These results are summarized
below.

Throughout the paper, all logarithms are in the natural basis, and
an event, defined for every $n$, is said to occur {\em with high
probability} or {\em almost surely} if its probability tends to $1$
as $n\to\infty$.

\subsection{Results for expanders with a non-fixed girth}
A simple argument shows that the requirement for a non-fixed girth
is essentially necessary if one wishes that the number of visits at vertices
will exhibit a Poisson distribution. Indeed, if $G$ is a graph where
every vertex is contained in a cycle of a fixed length, the probability that
the walk $\widetilde{W}$ will traverse a cycle $t$ consecutive times, becomes much larger than the Poisson
probability of $t$ visits to a
vertex for a sufficiently large $t$.
Simple random walks correspond to the case of cycles of
length 2.

On the other hand, this requirement on the girth turns out to be
sufficient as-well: as long as the girth of an $(n,d,\lambda)$ graph
$G$ tends to infinity arbitrarily slowly with $n$, the number of
visits that a non-backtracking random walk makes to vertices
exhibits a Poisson distribution, as the following theorem states:

\begin{theorem}\label{thm-count-fixed}
For fixed $d \geq 3$ and fixed $\lambda < d$, let $G$ be an
$(n,d,\lambda)$ graph whose girth is at least $g=\omega(1)$. Let
$\widetilde{W}=(w_0,\ldots,w_n)$ denote a non-backtracking random
walk of length $n$ on $G$ from $w_0$, and $N_t$ the number of
vertices which $\widetilde{W}$ visits precisely $t$ times:
\begin{equation}\label{eq-Nt-def} N_t  =
\left|\big\{ v \in V(G)~:~\left|\{1\leq i\leq n:w_i=v\}\right|=t
\big\}\right|~.\end{equation} Then for every fixed $t$, $N_t/n=
1/(\mathrm{e}t!)+o(1)$ almost surely, the $o(1)$-term tending to 0
as $n\to\infty$.
\end{theorem}
As we later mention, the above theorem in fact holds even for
non-fixed $\lambda,d$ (as long as the spectral gap is large compared
to $d/g$). The essence of the proof of Theorem \ref{thm-count-fixed}
lies in proving that the variables counting the visits at vertices,
whose pairwise distances are large, are asymptotically independent
Poisson variables:
\begin{proposition}
  \label{prop-poisson-fixed}
Let $G$ be a graph as in Theorem \ref{thm-count-fixed}. For some
fixed $r$ and $\mu > 0$, let $v_1,\ldots,v_r$ denote vertices of $G$
whose pairwise distances are at least $g$. Let $\widetilde{W}$ be a
non-backtracking random walk of length $m=\mu n$ on $G$ starting
from $v_1$, and $X_i$ be the number of visits that $\widetilde{W}$
makes to $v_i$. Then
$(X_1,\ldots,X_r)\stackrel{d}{\to}(Z_1,\ldots,Z_r)$ as $n\to\infty$,
where the $Z_i$-s are i.i.d. Poisson random variables with means
$\mu$, $Z_i\sim\Po(\mu)$.
\end{proposition}
\begin{remark}\label{rem-girth-req}
The statement of Proposition \ref{prop-poisson-fixed} holds (with
the same proof) even if we replace the requirement on the girth of
$G$ with the weaker assumption, that $v_1,\ldots,v_r$ are not
contained in a closed nontrivial walk of length smaller than $g$. In this case, the
presence of other possibly short cycles in $G$ has no effect on the
limiting distribution of $(X_1,\ldots,X_r)$.
\end{remark}
\begin{remark} The parameters $\lambda$ and $d$ in Theorem \ref{thm-count-fixed}
and in Proposition \ref{prop-poisson-fixed} need not be fixed, as long
as the spectral gap, $d-\lambda$, is $\omega(d/g)$ (where
$g=\omega(1)$ was the lower bound on the girth of $G$). For the sake
of simplicity, we prove the case of fixed $d,\lambda$, and later
describe the required adjustments for the general case.
\end{remark}

\subsection{Stronger results for high-girth expanders}
Letting $N_t$ continue to denote the number of vertices visited $t$
times by a non-backtracking random walk of length $n$ on $G$, as in
\eqref{eq-Nt-def}, consider the case where $t$ is no longer fixed.
In case we want to extend the result of Theorem
\ref{thm-count-fixed} for values of $t$ which depend on $n$, and
approximate $N_t$ by $n/(\mathrm{e}t!)$ uniformly over all $t$, a
behavior analogous to the balls and bins model, we need to assume a
larger girth. As noted in \cite{ABLS}, there are $d$-regular
expander graphs with girth $g$, where a typical non-backtracking
random walk visits some vertex $\Omega(\log n / g)$ times.
Therefore, the girth should be at least $\Omega(\log\log n)$ to
allow the number of visits to exhibit a Poisson distribution for all
$t$.

Indeed, an $\Omega(\log\log n)$ girth suffices in order to
approximate the above number of visits uniformly over all values of
$t$ up to the asymptotically maximal number of visits:

\begin{theorem}\label{thm-count}
Let $G$ be as in Theorem \ref{thm-count-fixed}. If $g > 10
\log_{d-1}\log n$ then for every fixed $\delta > 0$, the following
holds with high probability:
\begin{equation}\begin{array}{cl}
\left|\displaystyle{\frac{N_t}{n/(\mathrm{e}t!)}} - 1\right| \leq
\frac{1}{\log\log\log n}& \mbox{ for all }t < F(1-\delta)~, \\
\noalign{\medskip}
 N_t = 0& \mbox{ for all }t  > F(1+\delta)~,
 \end{array}
 \end{equation}
 where $\{N_t\}$ are the variables defined in \eqref{eq-Nt-def} and $F(x) = \left(1 +
x\frac{\log\log\log n}{\log\log n}\right)\frac{\log n}{\log\log n}$.
\end{theorem}
The above theorem is also related to the notion of conserving random
bits, which was mentioned earlier (where the goal was to simulate a
distribution which had an exponentially small probability of
avoiding a set, using a linear number of random bits). Indeed, using
$\Theta(n)$ random bits, it is possible to simulate a distribution
which resembles the resulting distribution of throwing $n$ balls to
$n$ bins uniformly and independently, as opposed to the naive
approach, which requires $n \log n$ random bits.

Theorem \ref{thm-count} also immediately gives the result of
\cite{ABLS} regarding the maximal number of visits that a
non-backtracking random walk makes to a single vertex, with an
improved threshold window, replacing the $o\left(\frac{\log
n}{\log\log n}\right)$ error term by $o\left(\frac{(\log
n)(\log\log\log n)}{(\log\log n)^2}\right)$:
\begin{corollary}\label{cor-max} For any fixed $d \geq 3$ and fixed $\lambda < d$
the following holds: if $G$ is an $(n,d,\lambda)$ graph whose girth
is larger than $10\log_{d-1}\log n$, then the maximal number of
visits to a single vertex made by a non-backtracking random walk of
length $n$ on $G$ is with high probability $$
\left(1+(1+o(1))\frac{\log\log\log n}{\log \log n}\right)\frac{\log
n}{\log\log n} ~,$$ where the $o(1)$-term tends to $0$ as
$n\to\infty$.
\end{corollary}
The proof of Theorem \ref{thm-count} follows from a result analogous
to Proposition \ref{prop-poisson-fixed}, however, since we require
an estimate on the rate of convergence to Poisson variables, we
require an extended form of a multivariate Brun's Sieve (Proposition
\ref{prop-multi-brun}). The proof may be applied to a more general
setting, where the set of visited vertices has size depending on
$n$, and the walk is of length $\omega(n)$. However, for the sake of
simplicity, we work in the setting of Theorem \ref{thm-count}, that
is, a fixed set of vertices and a walk of length $\Theta(n)$, as
stated in the following proposition:
\begin{proposition}
  \label{prop-poisson}
  Let $G$ be a graph as in Theorem \ref{thm-count-fixed}. For some
fixed $r$ and $\mu > 0$, let $v_1,\ldots,v_r$ denote vertices of $G$
whose pairwise distances are at least $g$. Let $\widetilde{W}$ be a
non-backtracking random walk of length $m=\mu n$ on $G$ starting
from $v_1$, and $X_i$ be the number of visits that $\widetilde{W}$
makes to $v_i$. If $g \geq c\log_{d-1}\log
  n$ for some fixed $c > 6$ then
$$\bigg|\frac{\Pr[\bigcap_{i=1}^r X_i = t_i]}{\prod_{i=1}^r
\Pr[Z = t_i]} - 1\bigg| \leq O\left((\log n)^{\frac{6-c}{4}}\right)
\mbox{ for all }t_1,\ldots,t_r \in \{0,1,\ldots,\lfloor\log
n\rfloor\}~,$$ where $Z \sim \Po(\mu)$ and the $o(1)$-term tends to
$0$ as $n\to\infty$.
\end{proposition}
\begin{remark} As in the case of Proposition \ref{prop-poisson-fixed} (see Remark \ref{rem-girth-req}), the
requirement on the girth in Proposition \ref{prop-poisson} may be
replaced with the assumption that $\{v_1,\ldots,v_r\}$ are not
contained in a closed nontrivial walk of length smaller than $g$.
\end{remark}

\subsection{Organization}
The rest of the paper is organized as follows: in Section
\ref{sec::const} we prove Theorem \ref{thm-count-fixed} and
Proposition \ref{prop-poisson-fixed} concerning expanders with a
non-fixed girth. In Section \ref{sec::brun} we formulate and prove
the multivariate version of Brun's Sieve which specifies the rate of
convergence to the limiting distribution. This version is
subsequently used in Section \ref{sec::high-girth} to prove the
above mentioned Theorem \ref{thm-count} and Proposition
\ref{prop-poisson}. The final section, Section
\ref{sec::conclusion}, is devoted to concluding remarks and some
open problems.

\section{A Poisson approximation for expanders with a non-fixed girth}\label{sec::const}
\subsection{Proof of Proposition \ref{prop-poisson-fixed}}
For the simpler goal of proving Poisson convergence without
estimating its rate, we will need the following known results. The
well known univariate version of Brun's Sieve states the following:
\begin{theorem}[Brun's Sieve]\label{thm-simple-brun} Let $X=X(n)$ be a sum of indicator
variables, and let $\mu > 0$. If for every $r$,
$\lim_{n\to\infty}\mathbb{E}\binom{X}{r} = \mu^r / r!$, then
$X\stackrel{d}{\to} Z$ as $n\to\infty$, where $Z \sim \Po(\mu)$.
\end{theorem}
See, e.g., \cite{ProbMethod} (pp. 119-122) for the derivation of
this result from the Inclusion-Exclusion Principle and the
Bonferroni inequalities \cite{Bonferroni}, as well as for several
applications. A multivariate version of Brun's Sieve is stated
in \cite{JLR}, and a proof of the multivariate version by
induction (using Brun's Sieve once for the base of the induction,
and once more for the induction step) appears in \cite{Wormald}:
\begin{theorem}[Multivariate Brun's Sieve]\label{thm-simple-multibrun}
Let $X_1=X(n),\ldots,X_r=X_r(n)$ denote sums of indicator variables,
and let $\mu_1,\ldots,\mu_r > 0$. If for every $t_1,\ldots,t_r$,
$\lim_{n\to\infty}\mathbb{E}[\prod_{i=1}^r\binom{X_i}{t_i}] =
\prod_{i=1}^r \mu_i^{t_i} / {t_i}! $, then $(X_1,\ldots,X_r)
\stackrel{d}{\to} (Z_1,\ldots,Z_r)$, where the $Z_i$-s are
independent Poisson variables, $Z_i \sim \Po(\mu_i)$.
\end{theorem}
While a stronger version of Brun's Sieve is proved in Section
\ref{sec::brun} (Proposition \ref{prop-multi-brun}), Theorem
\ref{thm-simple-multibrun} suffices for the proof of Proposition
\ref{prop-poisson-fixed}, where the rate of convergence to the
Poisson distribution is not specified. Indeed, letting
$X_1,\ldots,X_r$ and $\mu$ be as in Proposition
\ref{prop-poisson-fixed}, we need to prove that
\begin{equation}
  \label{eq-simple-brun-req}
\lim_{n\to\infty}\mathbb{E}[\prod_{i=1}^r\binom{X_i}{t_i}] =
\prod_{i=1}^r \frac{\mu^{t_i}}{{t_i}!} ~\mbox{ for all
}t_1,\ldots,t_r~.
\end{equation}
Fix integers $t_1,\ldots,t_r$, and set $t = \sum_{i=1}^r t_i$. Let
$(v_1,w_1,\ldots,w_m)$ denote the path of the non-backtracking
random walk $\widetilde{W}$, and for all $i\in[r]$ and $j\in[m]$ let
$X_{i j}$ denote the indicator for the event that $\widetilde{W}$
visits $v_i$ in position $j$; that is, $X_{i
j}=\mathbf{1}_{\{w_j=v_i\}}$, and by definition, $X_i = \sum_{j=1}^m
X_{i j}$. It follows that:
\begin{equation}
  \label{eq-factorial-moment-def}
  \mathbb{E}\Big[\prod_{i=1}^{r}\binom{X_i}{t_i}\Big] = \mathop{\textstyle{\mbox{\Large$\sum'$}}}_{I_1,\ldots,I_r}
  \Pr\Big[\bigcap_{i\in[r]} \bigcap_{j \in I_i}X_{i j}=1\Big]~,
\end{equation}
where $\sum'$ ranges over $I_1,\ldots,I_r \subset [m]$ with $|I_i| =
t_i$. We will rewrite the right-hand-side of the above equation. To
this end, set
\begin{equation}
  \label{eq-L-def}L = (\log n)^2~,
\end{equation} and let $g=\omega(1)$ be a lower bound for the girth of $G$, which
satisfies $g = o(L)$ (such a $g$ exists by the assumption on $G$).
For all $s \in \{0,\ldots,t=\sum t_i\}$, let $\mathcal{I}_s$ denote
the collection of $r$-tuples $(I_1,\ldots,I_r)$ where:
\begin{itemize}
\item $I_1,\ldots,I_r$ are {\em disjoint} subsets of $[m]$ and $|I_j|=t_j$ for all $j$.
\item There are precisely $s$ consecutive elements of $\cup_j I_j \cup \{0\}$ whose distance is less than $L$.
\end{itemize}
 In other words:
\begin{equation}\label{eq-calI-def}
\mathcal{I}_s = \left\{ (I_1,\ldots,I_r) ~:~
\begin{array}{l}\bigcup_j I_j = \{x_1,x_2,\ldots,x_t\} \subset [m],~x_0 = 0,~|I_j|=t_j \mbox{ for all $j$,}\\
x_{i-1}< x_i \mbox{ for all $i$, and } |\{1 \leq i \leq
t:x_i-x_{i-1} < L\}|=s\end{array} \right\}~.\end{equation} Notice
that the events $X_{i j} = 1$ and $X_{i' j} = 1$ are disjoint for $i
\neq i'$. Therefore, \eqref{eq-factorial-moment-def} takes the
following form:
\begin{equation}
  \mathbb{E}\Big[\prod_{i=1}^{r}\binom{X_i}{t_i}\Big] =
  \sum_{s=0}^t \sum_{(I_1,\ldots,I_r)\in \mathcal{I}_s}
\Pr[\bigcap_{i\in[r]} \bigcap_{j\in I_i} X_{i j}=1]~.
\end{equation}
The following claim estimates the probability that a
non-backtracking random walk, starting from some given $v_i$, would
end up in some given $v_j$ after less than $L$ steps, as well as
after some given $k \geq L$ number of steps. Here and in what
follows, the notation $\widetilde{P}_{u v}^{(k)}$ denotes the
probability that a non-backtracking random walk of length $k$, which
starts in $u$, ends in $v$.
\begin{claim}\label{cl-visit-bound} Let $G$ be as above, and define:
$M=\max_{i,j \in [r]}\sum_{k<L} \widetilde{P}_{v_i v_j}^{(k)}$. Then
$M = o(1)$ and $\widetilde{P}_{v_i v_j}^{(k)}=\frac{1+o(1)}{n}$ for
all $k \geq L$ and $i,j\in[r]$, where in both cases the $o(1)$-term
tends to $0$ as $n\to\infty$.
\end{claim}
\begin{proof}
We need a few results on the mixing of non-backtracking random
walks, proved in \cite{ABLS}. The {\em mixing-rate} of a
non-backtracking random walk on $G$ is defined as:
\begin{equation}\label{eq-rho-def}\rho(G) = \limsup_{k\to\infty} \max_{u,v\in V}
\big|\widetilde{P}_{u v}^{(k)} -
\frac{1}{n}\big|^{1/k}~.\end{equation} Theorem 1.1 of \cite{ABLS}
determines the value of $\rho$ as a function of $\lambda$ and $d$:
\begin{equation}\label{eq-rho-val}  \rho=
\frac{\psi\left(\frac{\lambda}{2\sqrt{d-1}}\right)}{\sqrt{d-1}}~,\mbox{
where } \psi(x) = \left\{\begin{array}
  {ll}x+\sqrt{x^2-1} & \mbox{If }~x \geq 1~,\\
1 & \mbox{If }~0 \leq x \leq 1~.
\end{array}\right.\end{equation}
As shown in \cite{ABLS}, one can verify that $\rho \leq
\max\{\frac{\lambda}{d},\frac{1}{\sqrt{d-1}} \}$, and in our case,
as $\lambda$ and $d$ are both fixed, so is $0 < \rho < 1$.
Furthermore, by the proof of the above theorem,
\begin{equation}\label{eq-rho-convergence} \max_{u v} \big|\widetilde{P}_{u v}^{(k)} - \frac{1}{n}\big|
\leq (1+o(1))\rho^k ~,\end{equation} where the $o(1)$-term tends to
$0$ as $k\to\infty$, and is independent of $n$. In particular, by
the choice of $L$ to be $(\log n)^2$, for every sufficiently large
$n$ we have
\begin{equation}\label{eq-L-mixing} \widetilde{P}_{u v}^{(k)} =
\frac{1+o(1)}{n}~\mbox{ for all $k\geq L$ and all $u,v$.}
\end{equation}
Take some $i,j\in[r]$ (not necessarily distinct). By the assumption
on the pairwise distances of $v_1,\ldots,v_r$ and the girth of $G$,
we have $\widetilde{P}_{v_i v_j}^{(k)}=0$ for all $k < g$. On the
other hand, \eqref{eq-rho-convergence} and the fact that
$g=\omega(1)$ imply that $\widetilde{P}^{(k)}_{u v} \leq \frac{1}{n}
+ (1+o(1))\rho^k$ for all $k \geq g$, giving the upper bound:
$$\sum_{k=1}^{L-1}\widetilde{P}_{v_i v_j}^{(k)} =
\sum_{k=g}^{L-1} \widetilde{P}^{(k)}_{v_i v_j} \leq \frac{L-g}{n} +
\frac{(1+o(1))\rho^g}{1-\rho} ~.$$ The required result now follows
from the fact that $L=o(n)$, $\rho$ is fixed and $g=\omega(1)$.
\end{proof}
For convenience, when examining some element $(I_1,\ldots,I_r)\in
\mathcal{I}_s$, we use the following notation:
denote by $i_1,\ldots,i_s\in[m]$ the $s$ indices of the $x_i$-s
which satisfy $|x_i-x_{i-1}|<L$, as in \eqref{eq-calI-def},
the definition of $\mathcal{I}_s$. In addition, for every $i\in[m]$,
let $v(x_i)$ denote the vertex $v_j$, where $j\in[r]$ is the single
index satisfying $x_i\in I_j$. The following holds:
\begin{align}
& \sum_{(I_1,\ldots,I_r)\in \mathcal{I}_s}
\Pr[\bigcap_{i\in[r]} \bigcap_{j\in I_i} X_{i j}=1] \nonumber \\
& \leq  \binom{t}{t_1,\ldots,t_r}\binom{m}{t-s}\binom{t}{s}
\left(\frac{1+o(1)}{n}\right)^{t-s}
\sum_{k_1=1}^{L-1}\ldots\sum_{k_s=1}^{L-1} \prod_{j=1}^{s} \widetilde{P}^{(k_j)}_{v(x_{i_j-1}) v(x_{i_j})} \nonumber \\
& \leq  \binom{t}{t_r,\ldots,t_r}\binom{m}{t-s}\binom{t}{s}
\left(\frac{1+o(1)}{n}\right)^{t-s} M^s~.
\label{eq-upper-bound-Irs}\end{align} Letting $\xi(s)$ denote the
right hand side of \eqref{eq-upper-bound-Irs}, it follows that for
all $s < t$:
$$ \frac{\xi(s+1)}{\xi(s)} =
\frac{(t-s)^2}{(m-t+s+1)(s+1)}\cdot \frac{ n}{1+o(1)}\cdot M =
\Theta(M)=o(1)~.$$ We deduce that
\begin{align}
\mathbb{E}\Big[\prod_{i=1}^{r}\binom{X_i}{t_i}\Big] & =
  \sum_{s=0}^{t} \sum_{(I_1,\ldots,I_r)\in \mathcal{I}_s}
\Pr[\bigcap_{i\in[r]} \bigcap_{j\in I_i} X_{i j}=1] \nonumber\\
& \leq \sum_{s=0}^{t}\xi(s) \leq
(1+o(1))\xi(0) =(1+o(1)\binom{t}{t_1,\ldots,t_r}\binom{m}{t}\left(\frac{1+o(1)}{n}\right)^t \nonumber \\
 & = (1+o(1))\frac{\mu^t}{\prod_{i=1}^r t_i!}
~.\label{eq-factorial-moment-upper}\end{align} For the other
direction, consider $\mathcal{J}$, the collection of all $r$-tuples
of disjoint subsets of $[m]\setminus [L]$, $(I_1,\ldots,I_r)$, where
$|I_j|=t_j$ and the pairwise distances of the indices all exceed
$L$:
\begin{equation}\label{eq-calJ-def}
\mathcal{J} = \left\{ (I_1,\ldots,I_r)~:~\begin{array}{l}\bigcup_j
I_j = \{x_1,x_2,\ldots,x_t\} \subset \{L+1,\ldots,m\},\\
|I_j|=t_j \mbox{ for all $j$ and }x_{i} > x_{i-1} + L\mbox{ for all
$i$}\end{array} \right\}~.\end{equation} Since $\mathcal{J} \subset
\bigcup_s\mathcal{I}_s$ and $\widetilde{P}_{v_i
v_j}^{(k)}=\frac{1+o(1)}{n}$ for all $k \geq L$ and $i,j\in[r]$
(Claim \ref{cl-visit-bound}), we get: \begin{align}
\mathbb{E}\Big[\prod_{i=1}^{r}\binom{X_i}{t_i}\Big] & \geq
  \sum_{(I_1,\ldots,I_r)\in\mathcal{J}}
  \left(\frac{1-o(1)}{n}\right)^t
  = \binom{t}{t_1,\ldots,t_r} \binom{m-L t}{t}
\left(\frac{1-o(1)}{n}\right)^t \nonumber\\
& = (1+o(1))\frac{\mu^t}{\prod_{i=1}^r
t_i!}~.\label{eq-factorial-moment-lower}\end{align} Inequalities
\eqref{eq-factorial-moment-upper} and
\eqref{eq-factorial-moment-lower} imply that
\eqref{eq-simple-brun-req} holds, completing the proof of
Proposition \ref{prop-poisson-fixed}. \qed

\begin{remark}
The assumption that $G$ is an $(n,d,\lambda)$ graph for some fixed
$d\geq 3$ and fixed $\lambda$ was exploited solely in Claim
\ref{cl-visit-bound}. In fact, the proof holds whenever for some
$L=o(n)$ and $g=\omega(1)$, $g < L$, the girth of $G$ is at least
$g$ and $\rho^g = o(1)$. Suppose that $d \geq 3$ but $\lambda,d$ are
no longer fixed. Recalling that $\rho \leq
\max\{\frac{\lambda}{d},\frac{1}{\sqrt{d-1}} \}$, the requirements
of Proposition \ref{prop-poisson-fixed} may be replaced, for
instance, with $G$ being an $(n,d,\lambda)$ graph of girth larger
than $g=\omega(1)$, where $d-\lambda = \omega(d/g)$.
\end{remark}

\subsection{Proof of Theorem \ref{thm-count-fixed}}
To prove the theorem, we use the estimates given by Proposition
\ref{prop-poisson-fixed} for the cases $r=1,2$, and apply a simple
second moment argument. The assumptions of the theorem imply that
for any two vertices $u,v \in V$, whose distance, as well as their
distance from $w_0$, are all at least $g$, we have:
\begin{align}
&\Pr[X_u = t] = \frac{1}{\mathrm{e}t!} + o(1)~, \label{eq-prob-r=1} \\
 &\Pr[X_u = X_v = t ] = \Pr[X_u = t]^2 +
o(1) \mbox{ for every fixed }t~,\label{eq-prob-r=2}
\end{align}
where the two $o(1)$-terms tend to $0$ as $n\to\infty$. Let $g=g(n)$
be such that the girth of $G$ is at least $g$, and in addition, $g =
o(\log n)$. Let $N_t$ denote the number of vertices which
$\widetilde{W}$ visits precisely $t$ times; we wish to obtain an
estimate on the probability that $N_t =
(1+o(1))n/(\mathrm{e}t!)$. As $t$ is fixed, the effect of any $o(n)$
positions along $\widetilde{W}$ have on this value is negligible,
and we may ignore the set of vertices whose distance from $w_0$ is less
 than $g$. Therefore, let $U$ denote the set of vertices whose distance from $w_0$
 is at least $g$, and define
$$ N'_t = \sum_{u \in U} \mathbf{1}_{\{X_u=t\}}~.$$
Since $ |U| \geq n - d (d-1)^{g-1} = (1-o(1))n$, we have:
\begin{equation}\label{eq-Nt'-Nt-relation} \frac{| N_t - N'_t |}{n/(\mathrm{e}t!)}
= o(1)~,\end{equation} and thus, showing that $N'_t =
(1+o(1))n/(\mathrm{e}t!)$ almost surely will complete the proof. By
\eqref{eq-prob-r=1}, $$\mathbb{E}N'_t =
(1+o(1))\frac{n}{\mathrm{e}t!}~,$$ and denoting by $\delta(u,v)$ the
distance between two vertices $u,v$, we deduce the following from
\eqref{eq-prob-r=2}:
\begin{align}\var(N'_t) & \leq \mathbb{E}N'_t + \sum_{u \in U}\sum_{v \in U}
\Big(\Pr[X_u=X_v=t]-\Pr[X_u=t]\Pr[X_v=t] \Big)\nonumber\\
& \leq  \mathbb{E}N'_t + \bigg(\sum_{u \in U} \mathop{\sum_{v \in
U}}_{\delta(u,v) < g} \Pr[X_u=t]\bigg) +
o(n^2) \nonumber \\
& \leq \left(1 + o(n)\right)\mathbb{E}N'_t + o(n^2)= o(n^2)~.
\end{align}
Chebyshev's inequality now gives:
$$ \Pr \left[\left|N'_t - \frac{n}{\mathrm{e}t!}\right| = \Omega(n)\right]  = O\left(\var(N'_t)/n^2\right) =
o(1)~,$$ completing the proof of the theorem. \qed

\section{Multivariate Brun's Sieve with an estimated rate of
convergence}\label{sec::brun} Recall that the versions of Brun's
Sieve stated in Section \ref{sec::const} (Theorem
\ref{thm-simple-brun} and Theorem \ref{thm-simple-multibrun}) do not
specify the rate of convergence to the Poisson distribution, and
furthermore, the inductive proof of the multivariate case (which
appears in \cite{Wormald}) gives an undesirable extra dependence of
the rate of convergence on the number of variables. We therefore
prove the next version of Brun's Sieve, which follows directly from
a multivariate version of the Bonferroni inequalities:
\begin{proposition}\label{prop-multi-brun}
Let $\mathcal{A}_i = \{A_{ij} : j \in [M_i]\}$, $i\in[r]$, denote
$r$ classes of events, and denote by $X_i =
\sum_{j=1}^{M_i}\mathbf{1}_{A_{ij}}$ the number of events in
$\mathcal{A}_i$ which occur. Suppose that for some integer $T$ and
some choice of $\epsilon,s,\mu_1,\ldots,\mu_r > 0$ satisfying
 $s > \mu$ and $2\frac{\mu^s}{s!} < \epsilon <
(2{r\mathrm{e}^{\mu}})^{-2}$, where $\mu=\max_i|\mu_i|$,
 we have:
\begin{equation}\label{eq-brun-req}
\bigg|\frac{\mathbb{E}[\prod_{i=1}^r\binom{X_i}{t_i}]}{\prod_{i=1}^r
\mu_i^{t_i} / {t_i}!} - 1\bigg| \leq \epsilon \mbox{ for all }
t_1,\ldots,t_r \in \{0,1,\ldots,r (T+2s)\}~.
\end{equation}
Then:
\begin{equation}\label{eq-brun-stat}
\left|\frac{\Pr[\bigcap_{i=1}^r X_i = t_i]}{\prod_{i=1}^r\Pr[Z_i =
t_i]} - 1\right| \leq \epsilon' \mbox{ for all }t_1,\ldots,t_r \in
\{0,\ldots,T\}~,
\end{equation}
where $\epsilon' = 2\exp(2\sum_i \mu_i)\epsilon + \sqrt{\epsilon} $
and $Z_1,\ldots,Z_r$ are i.i.d., $Z_i \sim \Po(\mu_i)$.
\end{proposition}
\begin{proof}
We need the following known multivariate generalization of the
Bonferroni inequalities:
\begin{theorem}[\cite{Meyer}]\label{thm-gen-bon}
Let $\mathcal{A}_i = \{A_{i j} : j \in [M_i]\}$, $i\in[r]$, denote
$r$ classes of events, and let $X_i = \sum_{j=1}^{M_i}
\mathbf{1}_{A_{i j}}$ denote the number of events in $\mathcal{A}_i$
which occur. Define:
\begin{equation}\label{eq-S-def}
S^{(i_1,\ldots,i_r)} = \mathbb{E}[\prod_j \binom{X_j}{i_j}] =
\mathop{\sum_{I_1\subset [M_1]}}_{|I_1| = i_1}\ldots
\mathop{\sum_{I_r\subset [M_r]}}_{|I_r| = i_r} \Pr[\bigcap_{j=1}^r
\bigcap_{k \in I_j } A_{j k}]\end{equation} The following holds for
all non-negative integers $m_1,\ldots,m_r$, $0 \leq m_j \leq M_j$,
and $k \geq 0$:
\begin{equation}\label{eq-bon-bounds}
\begin{array}{l}\Lambda(2k+1) \leq \Pr[\cap_i X_i = m_i] \leq
\Lambda(2k)~,\mbox{ where:}\\
\noalign{\medskip} \displaystyle{\Lambda(k) = \sum_{t=\sum
m_j}^{(\sum m_j) + k} \sum_{\sum i_j=t}(-1)^{t-\sum m_j}
\bigg(\prod_{j=1}^r \binom{i_j}{m_j}\bigg) S^{(i_1,\ldots,i_r)}}~.
\end{array}
\end{equation}\end{theorem}
As the function $f(x,k)=\sum_{l=0}^k x^l/l!$ satisfies $|f(x,k) -
\mathrm{e}^x| \leq 2 \frac{|x|^k}{k!}$ for all $x$ with $|x| \leq
\frac{k+1}{2}$, the assumption on $s$ implies that
\begin{equation}
  \label{eq-mu-sum-convergence}
 \left| \sum_{l=0}^k \frac{x^l}{l!} - \mathrm{e}^x \right| \leq
 \epsilon ~\mbox{ for all $k \geq 2s-1$ and }|x| \leq \mu~.
\end{equation}
 Let $m_1,\ldots,m_r \in \{0,\ldots,T\}$, and set $M=\sum_i m_i$ and
 \begin{equation}\label{eq-p-def}p = \Pr[\cap_i Z_i = m_i] = \prod_{i=1}^r \mathrm{e}^{-\mu_i}
 \frac{\mu_i^{m_i}}{m_i!}~.
 \end{equation}
Notice that $r(T+2s) \leq \min_i M_i$, otherwise we would get the
contradiction $\epsilon \geq 1$ from \eqref{eq-brun-req}. According
to the notations of Theorem \ref{thm-gen-bon}, by
\eqref{eq-brun-req} and the facts $M=\sum_{i=1}^r m_i$ and $m_i \leq
T$:
$$(1-\epsilon)\prod_j
\frac{\mu_j^{i_j}}{i_j!} \leq  S^{(i_1,\ldots,i_r)} \leq
(1+\epsilon)\prod_j \frac{\mu_j^{i_j}}{i_j!}\mbox{ for all
}i_1,\ldots,i_r \in \{0,\ldots,M+2rs\}~.
$$
Therefore, the following holds:
\begin{align}
 \Lambda(k) &= \sum_{t=M}^{M +
k} \sum_{\sum i_j=t}(-1)^{t-M} \bigg(\prod_{j=1}^r
\binom{i_j}{m_j}\bigg) S^{(i_1,\ldots,i_r)} \nonumber \\
& \leq  \sum_{t=M}^{M + k} \sum_{\sum i_j=t} \bigg(\prod_{j=1}^r
\binom{i_j}{m_j} \frac{\mu_j^{i_j}}{i_j!}\bigg) \left((-1)^{t-M} +\epsilon\right) \nonumber \\
& =  \bigg( \prod_{j=1}^r \sum_{i_j=m_j}^{M+k}
 \binom{i_j}{m_j} \frac{\mu_j^{i_j}}{i_j!}
(-1)^{i_j-m_j} \bigg) + \epsilon\bigg( \prod_{j=1}^r
\sum_{i_j=m_j}^{M+k}
 \binom{i_j}{m_j} \frac{\mu_j^{i_j}}{i_j!}
\bigg)
\nonumber\\
& - \sum_{i_1=m_1}^{M+k} \ldots \sum_{i_r=m_r}^{M+k}
\mathbf{1}_{\sum i_j > M+k} \bigg(\prod_{j=1}^r
\binom{i_j}{m_j}\frac{\mu_j^{i_j}}{i_j!}\bigg)\left( (-1)^{(\sum
i_j) - M} + \epsilon\right) ~. \label{eq-Lambda-bound}
\end{align}
Let $E_1,E_2,E_3$ denote the final three expressions in
\eqref{eq-Lambda-bound}, that is, $\Lambda(k) \leq E_1 + E_2 - E_3$.
A similar calculation gives $\Lambda(k) \geq E_1 - E_2 - E_3$ (with
room to spare, as we could have replaced $E_3$ by a smaller
expression by replacing $\epsilon$ by $-\epsilon$). We therefore
wish to provide bounds on $E_1,E_2,E_3$. For all $k \geq 2s-1$ we
have:
\begin{align}
\left| 1 - \frac{E_1}{p}  \right| & = \bigg| 1 -
\frac{1}{p}\prod_{j=1}^r \sum_{i_j=m_j}^{M+k}
 \binom{i_j}{m_j} \frac{\mu_j^{i_j}}{i_j!}
(-1)^{i_j-m_j}\bigg| \nonumber \\
&= \bigg| 1 - \frac{1}{p}\prod_{j=1}^r
\frac{\mu_j^{m_j}}{m_j!}\sum_{l=0}^{M-m_j+k} \frac{(-\mu_j)^l}{l!}
\bigg|  \leq (1+\epsilon \mathrm{e}^\mu)^r-1 \leq
\mathrm{e}^{\sqrt{\epsilon}/2}-1 \leq ~ \sqrt{\epsilon},
\label{eq-E1-bound}\end{align} where the first inequality is by
\eqref{eq-mu-sum-convergence}, as $\left|1 - \mathrm{e}^{\mu_j}
\sum_{l=0}^k \frac{(-\mu_j)^l}{l!} \right| \leq \epsilon
\mathrm{e}^{\mu}$, and the second follows from the assumption that
$\epsilon < (2r\mathrm{e}^\mu)^{-2}$. Similarly, for $k\geq 2s-1$:
\begin{align}
E_2 & = \epsilon\prod_{j=1}^r \sum_{i_j=m_j}^{M+k}
 \binom{i_j}{m_j} \frac{\mu_j^{i_j}}{i_j!}
 \leq \epsilon \prod_{j=1}^r \mathrm{e}^{\mu_j}
 \frac{\mu_j^{m_j}}{m_j!} =
\epsilon \exp\big(2\sum_i \mu_i\big) p~.
\label{eq-E2-bound}\end{align} For the bound on $|E_3|$, recall that
$M=\sum_i m_i$, and hence, if $\sum_j i_j \geq M + 2rs$ we must have
$i_t \geq m_t + 2s$ for some $t$. Therefore, for all $k \geq 2rs -
1$:
\begin{align}
|E_3| & =  \bigg |\sum_{i_1=m_1}^{M+k} \ldots \sum_{i_r=m_r}^{M+k}
\mathbf{1}_{\sum i_j > M+k} \bigg(\prod_{j=1}^r
\binom{i_j}{m_j}\frac{\mu_j^{i_j}}{i_j!}\bigg)\left( (-1)^{(\sum
i_j) - M} + \epsilon\right)
 \bigg|  \nonumber \\
& \leq (1+\epsilon)\sum_{i_1=m_1}^{M+k} \ldots \sum_{i_r=m_r}^{M+k}
\mathbf{1}_{\sum i_j > M+k} \prod_{j=1}^r
\binom{i_j}{m_j}\frac{\mu_j^{i_j}}{i_j!} \nonumber \\
& \leq (1+\epsilon)\sum_{t=1}^r \sum_{i_t=m_t+2s}^{M+k}
 \binom{i_t}{m_t} \frac{\mu_t^{i_t}}{i_t!} \bigg(\prod_{j\neq t}
\mathrm{e}^{\mu_j}\frac{\mu_j^{i_j}}{i_j!}\bigg) \leq \epsilon
\exp\big(2\sum_i \mu_i\big)p~,
  \label{eq-E3-bound}
\end{align}
where the last inequality is by the fact $\frac{\mu^s}{s!} <
\epsilon / 2$, which implies that
$$ (1+\epsilon)\sum_{i_t=m_t+2s}^{M+k} \binom{it}{m_t}\frac{\mu_t^{i_t}}{i_t!}
\leq (1+\epsilon) \frac{\mu_t^{m_t}}{m_t!}\sum_{l \geq 2s}
\frac{\mu_t^l}{l!}  \leq (1+\epsilon)\frac{\mu_t^{m_t}}{m_t!}\cdot 2
\frac{\mu_t^{2s}}{(2s)!} \leq
\frac{1+\epsilon}{2}\epsilon^2\mathrm{e}^{\mu_t} \leq
\frac{\epsilon}{r} \mathrm{e}^{\mu_t}~.$$
 Altogether, combining \eqref{eq-E1-bound},
\eqref{eq-E2-bound} and \eqref{eq-E3-bound} we get the following for
$k \geq 2rs-1$:
$$ \left| \frac{\Lambda(k)}{p} - 1 \right| \leq \sqrt{\epsilon} + 2\exp\big(2\sum_i\mu_i\big)\epsilon
= \epsilon'~.$$ The proof is completed by the fact that $
\Lambda(2rs)\leq \Pr[\cap_i X_i=m_i] \leq \Lambda(2rs-1)$.
\end{proof}

\section{A Poisson approximation for high-girth expanders}\label{sec::high-girth}
In this section, we prove Proposition \ref{prop-poisson} and its
corollary, Theorem \ref{thm-count}, which are analogous to
Proposition \ref{prop-poisson-fixed} and Theorem
\ref{thm-count-fixed}, but also provide an estimate on the rate
 of convergence to the limiting distributions. This is imperative when
looking at vertices which are visited $t$ times, for $t$ tending to
$\infty$ with $n$.
 The proof of Proposition \ref{prop-poisson} follows the ideas of the proof of Proposition
\ref{prop-poisson-fixed}, where instead of the simple version of
Brun's Sieve, we use Proposition \ref{prop-multi-brun} proved in
Section \ref{sec::brun}.
\subsection{Proof of Proposition \ref{prop-poisson}}\label{sec::approx}
Recall that $g \geq c \log_{d-1}\log n$ for some fixed $c > 6$. We
need the following definitions:
\begin{eqnarray}
 \tau  &=& \min_t \left\{ \Big|\widetilde{P}^{(k)}_{u v} -
\frac{1}{n} \Big| \leq \frac{1}{n^2} ~\mbox{ for all $u,v\in V$ and
$k \geq
t$}\right\}~.\label{eq-fine-mixing-time-def} \\
 T &=& \lfloor \log n\rfloor~,\label{eq-T-def}\\
 h &=& (\log n)^{3-\frac{c}{2}}~.\label{eq-T-h-def}
\end{eqnarray}
Recalling \eqref{eq-rho-convergence}, for $k = \Omega(\log n)$ we
have $\big|\widetilde{P}_{u v}^{(k)} - \frac{1}{n}\big| \leq
n^{-\Omega(1)}$, giving
\begin{equation}\label{eq-tau-bound}\tau = O(\log n)~.\end{equation}
 According to the notation of Proposition
\ref{prop-multi-brun}, set $\mu_i=\mu$ for all $i$, let $h$ play the
role of $\epsilon$, and define $h'$ to be the analogue of
$\epsilon'$:
\begin{equation}\label{eq-h'-def}h' =
2\mathrm{e}^{2r\mu}h+\sqrt{h}=(1+o(1))(\log
n)^{\frac{6-c}{4}}~.\end{equation} It follows from Proposition
\ref{prop-multi-brun} that, in order to show that
$$ \left|\frac{\Pr[\bigcap_{i=1}^r X_i = t_i]}{\prod_{i=1}^r\Pr[Z =
t_i]} - 1\right| \leq O(h') \mbox{ for all }t_1,\ldots,t_r \in
\{0,\ldots,T\}~,$$ it suffices to show that for some $s$ satisfying
$s > \mu$ and $2\frac{\mu^s}{s!} < h < (2r\mathrm{e}^\mu)^{-2}$ we
have: \begin{equation}\label{eq-brun-req-revisited}
\bigg|\frac{\mathbb{E}[\prod_{i=1}^r\binom{X_i}{t_i}]}{\mu^t/\prod_{i=1}^r
{t_i}!} - 1\bigg| \leq O(h)\mbox{ for all } t_1,\ldots,t_r \in
\{0,1,\ldots,r (T+2s)\}~.\end{equation} Substituting $s = T$, the
requirements $T > \mu$ and $h < (2r\mathrm{e}^\mu)^{-2}$ immediately
hold for a sufficiently large $n$, as $T = \omega(1)$, $h=o(1)$ and
both $\mu$ and $r$ are fixed. The requirement $2\frac{\mu^T}{T!}<h$
holds as well, since $2\frac{\mu^T}{T!} = \exp\left(-(1-o(1))(\log
n)(\log\log n)\right)$, and for a sufficiently large $n$, this term
is clearly smaller than $h = \exp\left(-O(\log\log n)\right)$.
Therefore, proving \eqref{eq-brun-req-revisited} for $s=T$ would
complete the proof of the proposition, that is, we need to show that
 \begin{equation}\label{eq-brun-req-s=T}
\bigg|\frac{\mathbb{E}[\prod_{i=1}^r\binom{X_i}{t_i}]}{\mu^t/\prod_{i=1}^r
{t_i}!} - 1\bigg| \leq O(h) \mbox{ for all } t_1,\ldots,t_r \in
\{0,1,\ldots,3rT\}~.\end{equation} Let $t_1,\ldots,t_r \in
\{0,\ldots, 3rT\}$, and set
$$t = \sum_{i=1}^r t_i \leq 3r^2 T = O(\log n)~.$$
 Let $(v_1,w_1,\ldots,w_m)$ denote the path of the
non-backtracking random walk $\widetilde{W}$, and as before, for all
$i\in[r]$ and $j\in[m]$ let $X_{i j}$ denote the indicator for the
event that $\widetilde{W}$ visits $v_i$ in position $j$:
$$ X_{i j}=\mathbf{1}_{\{w_j=v_i\}}~,\quad X_i = \sum_{j=1}^m
X_{i j}~.$$ As in the proof of Proposition \ref{prop-poisson-fixed},
we next define the collection $\mathcal{I}_s$, this time for $L =
\tau = O(\log n)$. For all $s \in \{0,\ldots,t=\sum t_i\}$, let
$\mathcal{I}_s$ denote the collection of $r$-tuples
$(I_1,\ldots,I_r)$ of disjoint subsets of $[m]$, $|I_j|=t_j$, so
that there are precisely $s$ consecutive elements of $\cup_j I_j
\cup \{0\}$ whose distance is less than $\tau$:
\begin{equation}\label{eq-calI-def-2}
\mathcal{I}_s = \left\{ (I_1,\ldots,I_r) ~:~
\begin{array}{l}\bigcup_j I_j = \{x_1,x_2,\ldots,x_t\} \subset [m],~x_0 = 0,~|I_j|=t_j \mbox{ for all $j$,}\\
x_{i-1} < x_i \mbox{ for all $i$, and } |\{1 \leq i \leq
t:x_i-x_{i-1} < \tau \}|=s\end{array} \right\}~.\end{equation} The
facts that the events $X_{i j}= 1$ and $X_{i' j}=1$ are disjoint for
$i \neq i'$ implies that
\begin{equation}
  \label{eq-factorial-moment-2}
  \mathbb{E}\Big[\prod_{i=1}^{r}\binom{X_i}{t_i}\Big] =
    \sum_{(I_1,\ldots,I_r)\in \mathcal{I}_s}
\Pr[\bigcap_{i\in[r]} \bigcap_{j\in I_i} X_{i j}=1]~.
\end{equation}
By definition \eqref{eq-fine-mixing-time-def}, for all $v_i$, $v_j$
and $k \geq \tau$, $\widetilde{P}_{v_i v_j}^{(k)} \leq
n^{-1}+n^{-2}$. The following claim estimates the sum of the
probabilities $\widetilde{P}_{v_i v_j}^{(k)}$ over all $k < \tau$.
\begin{claim}\label{cl-visit-bound-2} Let $G$ be as above, and define
$M=\max_{i,j \in [r]}\sum_{k<\tau} \widetilde{P}_{v_i v_j}^{(k)}$.
Then $M = O\left((\log n)^{1-\frac{c}{2}}\right)$.
\end{claim}
\begin{proof} Let $v_i$ and $v_j$ denote two (not necessarily
distinct) elements of $\{v_1,\ldots,v_r\}$. By the assumption on the
pairwise distances of $v_1,\ldots,v_r$ and the girth of $G$,
$\widetilde{P}_{v_i v_j}^{(k)}=0$ for all $k < g$. It remains to
estimate $\sum_{k=g}^{\tau -1} \widetilde{P}_{v_i v_j}^{(k)}$.

Set $\ell = \lfloor \frac{g-1}{2}\rfloor$, and consider $U$, the set
of vertices of $G$ whose distance from $v_j$ is at most $\ell$.
Since the girth of $G$ is at least $g$, the induced subgraph of $G$
on $U$ is isomorphic to a $d$-regular tree. Next, examine a
non-backtracking walk of length $k \geq \ell$ from $u$ to $v$;
crucially, since the walk cannot backtrack, the last $\ell$ vertices
along the walk must form a path from a leaf of the above mentioned
tree, up to its root. In each of the $\ell$ steps along this path
there is a probability of $1-\frac{1}{d-1}$ to stray from the path,
hence $\widetilde{P}^{(k)}_{u v} \leq (d-1)^{-\ell}$. Altogether,
$$\sum_{k=g}^{\tau -1} \widetilde{P}_{v_i v_j}^{(k)} \leq
\frac{\tau-g}{(d-1)^{\lfloor(g-1)/2\rfloor}} = O\left((\log
n)^{1-\frac{c}{2}}\right)~,$$
 as
required.
\end{proof}
Letting $i_1,\ldots,i_s\in[m]$ denote the $s$ indices of the $x_i$-s
which satisfy $|x_i-x_{i-1}| <
\tau$ in the definition \eqref{eq-calI-def-2} of $\mathcal{I}_s$, and defining
$v(x_i)=v_j$, where $j$ is such that $x_i\in I_j$, we
have:
\begin{align}
& \sum_{(I_1,\ldots,I_r)\in \mathcal{I}_s}
\Pr[\bigcap_{i\in[r]} \bigcap_{j\in I_i} X_{i j}=1] \nonumber \\
& \leq  \binom{t}{t_1,\ldots,t_r}\binom{m}{t-s}\binom{t}{s}
\left(\frac{1+{n^{-1}}}{n}\right)^{t-s}
\sum_{k_1=1}^{\tau-1}\ldots\sum_{k_s=1}^{\tau-1} \prod_{j=1}^{s} \widetilde{P}^{(k_j)}_{v(x_{i_j-1}) v(x_{i_j})} \nonumber \\
& \leq  \binom{t}{t_r,\ldots,t_r}\binom{m}{t-s}\binom{t}{s}
\left(\frac{1+{n^{-1}}}{n}\right)^{t-s} M^s~.
\label{eq-upper-bound-Irs-2}\end{align} Let $\xi(s)$ denote the
right hand side of \eqref{eq-upper-bound-Irs-2}. Recalling that
$m=\Theta(n)$ and $t = O(\log n)$, the following holds for all $s <
t$:
\[ \frac{\xi(s+1)}{\xi(s)} = \frac{(t-s)^2}{(m-t+s+1)(s+1)}\cdot
\frac{ n}{1+n^{-1}}\cdot M = O(t^2 M)=O\left((\log
n)^{3-\frac{c}{2}}\right) = o(1)~,\] where the last equality is by
the fact that $c>6$. Combining this with the fact that, as
$t=n^{o(1)}$, $(1+n^{-1})^t = 1+O(n^{-1+o(1)})=1+o(h)$, we get:
\begin{align}\sum_{s=0}^{t}\xi(s) & \leq
\frac{\xi(0)}{1-O(t^2 M)}=
\left(1+ O(t^2 M)\right)\binom{t}{t_1,\ldots,t_r}\binom{m}{t}\left(\frac{1+n^{-1}}{n}\right)^t \nonumber \\
& \leq \left(1+O\left((\log n)^{3-\frac{c}{2}}\right)\right)
\frac{\mu^t}{\prod_{i=1}^r t_i!}
 \leq (1+O(h))\frac{\mu^t}{\prod_{i=1}^r t_i!}  ~,\nonumber\end{align} and:
\begin{equation}\label{eq-factorial-moment-upper-2}
  \mathbb{E}\Big[\prod_{i=1}^{r}\binom{X_i}{t_i}\Big] \leq
  \sum_{s=0}^{t} \sum_{(I_1,\ldots,I_r)\in \mathcal{I}_s}
\Pr[\bigcap_{i\in[r]} \bigcap_{j\in I_i} X_{i j}=1]
 \leq (1+O(h))\prod_{i=1}^r \frac{\mu^{t_i}}{t_i!}~.\end{equation}
For the other direction, consider $\mathcal{J}$, the collection of all $r$-tuples
of disjoint subsets of $[m]\setminus [\tau]$, $(I_1,\ldots,I_r)$, where $|I_j|=t_j$
and the pairwise distances of the indices all exceed $\tau$:
\begin{equation}\label{eq-calJ-def-1}
\mathcal{J} = \left\{ (I_1,\ldots,I_r)~:~\begin{array}{l}\bigcup_j
I_j = \{x_1,x_2,\ldots,x_t\} \subset \{\tau+1,\ldots,m\},\\
|I_j|=t_j \mbox{ for all $j$ and }x_{i+1} > x_i + \tau\mbox{ for all
$i$}\end{array} \right\}~.\end{equation} By the definition
\eqref{eq-fine-mixing-time-def} of $\tau$, and the fact that
$(1+n^{-1})^t = 1+O(n^{-1+o(1)})=1+o(h)$,
\begin{align} \mathbb{E}\Big[\prod_{i=1}^{r}\binom{X_i}{t_i}\Big]
&\geq
 \sum_{(I_1,\ldots,I_r)\in \mathcal{J}}
\Pr[\bigcap_{i\in[r]} \bigcap_{j\in I_i} X_{i j}=1]
 \geq   \sum_{(I_1,\ldots,I_r)\in\mathcal{J}}
 \left(\frac{1-n^{-1}}{n}\right)^t \nonumber\\
&= \binom{t}{t_1,\ldots,t_r} \binom{m-\tau t}{t}
\left(\frac{1-n^{-1}}{n}\right)^t =
\frac{(1-o(h))\mu^t}{\prod_{i=1}^r t_i!}~.
\label{eq-factorial-moment-lower-2} \end{align} Combining
\eqref{eq-factorial-moment-upper-2} and
\eqref{eq-factorial-moment-lower-2}, we obtain that
\eqref{eq-brun-req-s=T} holds for all $0 \leq t_1,\ldots,t_r \leq
3rT$, completing the proof. \qed

\subsection{Proof of Theorem \ref{thm-count}}
The proof will follow from Proposition \ref{prop-poisson} using a
second moment argument, in a manner analogous to Theorem
\ref{thm-count-fixed}. Let $g = 10 \log_{d-1}\log n$; by the
assumption on $G$, Proposition \ref{prop-poisson} implies that for
any two vertices $u,v \in V$, whose distance, as well as their
distance from $w_0$, are all at least $g$, we have:
\begin{align}
&\Big|\frac{\Pr[X_u = t]}{1/(\mathrm{e}t!)} - 1\Big| =
O\left(\frac{1}{\log n}\right) \mbox{ for all }t \leq \log
n~,\label{eq-prob-r=1-bound-2} \\
& \Big|\frac{\Pr[X_u = X_v = t]}{1/(\mathrm{e}t!)^2} - 1\Big| =
O\left(\frac{1}{\log n}\right) \mbox{ for all }t\leq \log
n~.\label{eq-prob-r=2-bound-2}
\end{align}
Let $t$ be some integer satisfying
\begin{equation}\label{eq-t-range}t \leq \left(1 + c\frac{\log\log\log n}{\log\log
n}\right)\frac{\log n}{\log\log n}\mbox{ for some }c < 1~,
\end{equation}
and let $N_t$ denote the number of vertices which $\widetilde{W}$
visits precisely $t$ times. We wish to obtain an estimate on the
probability that $N_t = (1+o(1))n/(\mathrm{e}t!)$. The above choice
of $t$ implies that:
\begin{equation}\label{eq-poisson-exp}
\frac{n}{\mathrm{e}t!} \geq \exp\left((1-c-o(1))\log n
\frac{\log\log\log n}{\log\log n}\right) = \exp\left((1-c)(\log
n)^{1-o(1)}\right)~.
\end{equation}
Hence, the effect of any $(\log n)^{O(1)}$ positions along
$\widetilde{W}$ on this value is negligible, and
 we may ignore the set of vertices whose distance from $w_0$ is less
 than $g$. Therefore, let $U$ denote the set of vertices whose distance from $w_0$
 is at least $g$, let $X_u$ ($u \in U$) denote the number of visits which $\widetilde{W}$ makes to
 $u$, and let $N_t' = \sum_{u \in U} \mathbf{1}_{\{X_u=t\}}$. According to this definition, we have:
\begin{equation}\label{eq-Nt'-Nt-relation-2} \frac{| N_t - N'_t |}{n/(\mathrm{e}t!)} =
\exp\left(-(1-c)(\log n)^{1-o(1)}\right)~,\end{equation} and it
remains to determine the behavior of $N'_t$. By
\eqref{eq-prob-r=1-bound-2},
$$\left|\frac{\mathbb{E}N'_t}{n/(\mathrm{e}t!)} - 1 \right|=
\bigg|\Big(\sum_{u\in U}\frac{\Pr[X_u = t]}{n/(\mathrm{e}t!)}\Big) -
1 \bigg| = O(1/\log n) ~,$$ and we deduce from
\eqref{eq-poisson-exp} that
$$ \mathbb{E}N'_t = (1-o(1))\frac{n}{\mathrm{e}t!} =\Omega\left(\exp\left((1-c)(\log
n)^{1-o(1)}\right)\right)~.$$ Furthermore, denoting by $\delta(u,v)$
the distance between two vertices $u,v$, the following holds:
\begin{align}\var(N'_t) & \leq \mathbb{E}N'_t + \sum_{u \in U}\sum_{v \in U}
\Big(\Pr[X_u=X_v=t]-\Pr[X_u=t]\Pr[X_v=t] \Big)\nonumber\\
& \leq  \mathbb{E}N'_t + \bigg(\sum_{u \in U} \mathop{\sum_{v \in
U}}_{\delta(u,v) < g} \Pr[X_u=t]\bigg) +
\bigg(\sum_{u \in U} \mathop{\sum_{v \in U}}_{\delta(u,v) \geq g}
O\Big(\frac{1}{\log n}\Big) \Pr[X_u=t]^2\bigg) \nonumber \\
& \leq \left(1 + (\log n)^{O(1)}\right)\mathbb{E}N'_t +
O\left(\frac{(\mathbb{E}N'_t)^2}{\log n}\right)  =
O\left(\frac{(\mathbb{E}N'_t)^2}{\log n}\right)~.
\end{align}
Set $h=h(n)=\log\log\log n$. Applying Chebyshev's inequality gives
the following:
$$ \Pr \Big[\Big|\frac{N'_t}{n/(\mathrm{e}t!)} - 1\Big| \geq \frac{1}{h} \Big]  \leq
\Pr \Big[|N'_t - \mathbb{E}N'_t| \geq
\frac{1}{2h}\cdot\frac{n}{\mathrm{e}t!} \Big] = O\left(
\frac{h^2}{\log n}\right)~,$$ and summing this probability for all
$t$ in the range specified in \eqref{eq-t-range} (containing
$O(\frac{\log n}{\log\log n})$ values) we deduce that with high
probability:
\begin{equation}\label{eq-N't-result} \left|\frac{N'_t}{n/(\mathrm{e}t!)} -
1\right| \leq \frac{1}{\log\log\log n}\mbox{ for all }t \leq
\left(1+c\frac{\log\log\log n}{\log\log n} \right)\frac{\log
n}{\log\log n}~. \end{equation} Recalling the relation between
$N'_t,N_t$ in \eqref{eq-Nt'-Nt-relation-2}, we obtain that when
replacing $N'_t$ by $N_t$, \eqref{eq-N't-result} holds as well. It
remains to show that with high probability, $N_t=0$ for all $t
>t_0$, where
\begin{equation}\label{eq-t0-def}t_0 = \left(1 + c\frac{\log\log\log n}{\log\log
n}\right)\frac{\log n}{\log\log n}\mbox{ for some }c > 1~.
\end{equation}
Let $u$ be a vertex of $G$, and let $X_u$ denote the number of
visits that $\widetilde{W}$ makes to $u$. Consider $\widetilde{W}'$,
a non-backtracking random walk of length $n$ on $G$ starting from
$u$. Proposition \ref{prop-poisson} (for the case of one variable
$v_1=u$) implies that:
$$ \Pr[X'_u = t] = \frac{1+o(1)}{\mathrm{e}t!}~\mbox{
for all }t \leq \log n~,$$ where $X'_u$ counts the number of visits
that $\widetilde{W}'$ makes to $u$. Clearly, the probability that
$X_u > t_0$ is bounded from above by the probability that $X'_u \geq
t_0$ (as we can always condition on the first visit to $u$).
Therefore:
$$ \Pr[X_u > t_0] \leq \Pr[X'_u \geq t_0] = 1 - \sum_{l < t_0}\Pr[X'_u=l] \leq
\frac{1+o(1)}{\mathrm{e}t_0!}~.$$ We deduce that the expected number
of vertices with more than $t_0$ visits satisfies:
\begin{align}\mathbb{E}|\{u \in
V:X_u > t_0\}| \leq (1+o(1))\frac{n}{\mathrm{e}t_0!} =
\exp\left((1-c)(\log n)^{1-o(1)}\right)=o(1)~.\nonumber\end{align}
This completes the proof of the theorem. \qed

\section{Concluding remarks and open
problems}\label{sec::conclusion}
\begin{itemize}
  \item We have shown that the distribution of the number
of visits at vertices made by a non-backtracking random walk of
length $n$ on $G$, a regular $n$-vertex expander of fixed degree and
large girth, tends to a Poisson distribution with mean $1$.
Furthermore, if the girth is $\Omega(\log\log n)$ we prove the
following concentration result: with high probability, the number of
vertices visited $t$ times is $(1+o(1))\frac{n}{\mathrm{e}t!}$
uniformly over all $t \leq (1-o(1))\frac{\log n}{\log\log n}$, and
$0$ for all $t \geq (1+o(1))\frac{\log n}{\log\log n}$ (in fact, the
threshold window we get is sharper by a factor of
$\frac{\log\log\log n}{\log\log n}$). In particular, we obtain an
alternative proof for the typical maximal number of visits to a
vertex in the above walk, and (slightly) improve upon the estimate of this
maximum in \cite{ABLS}.
\item The above result implies that the distribution of the visits at vertices made by
a non-backtracking random walk of length $n$ on an $n$-vertex
expander of high girth is asymptotically the same as the result of
throwing $n$ balls to $n$ bins independently and uniformly at
random.
\item The main tool in the proof is an extended
version of Brun's Sieve, which includes error estimates and may be
of independent interest. Combining this result with some additional
ideas, we show that the variables counting the number of visits to
vertices, which are sufficiently distant apart, are asymptotically
independent Poisson variables. This implies the required result on
the overall distribution of the number of visits at vertices.

\item Theorem \ref{thm-count} characterizes the distribution of visits
at vertices in non-backtracking random walk on a high-girth regular
expander. For such a graph on $n$ vertices, the values $N_t/n$
converge to $(1+o(1))/(\mathrm{e}t!)$, where $N_t$ is the number of
vertices visited precisely $t$ times in a walk of length $n$ as
above. Moreover, we show that the above convergence of
$\{\frac{N_t}{n}\}$ is uniform over all values of $t$ up to roughly
$\frac{\log n}{\log\log n}$, after which $N_t$ is almost surely $0$.

It seems interesting to investigate this distribution,
$(\frac{N_0}{n},\frac{N_1}{n},\ldots,\frac{N_n}{n})$, as a parameter
of general vertex transitive graphs, and determine it for additional
families of such graphs.

\item Corollary \ref{cor-max} determines that the maximum number of
visits to a vertex, made by a typical non-backtracking random walk
of length $n$ on a high-girth $n$-vertex regular expander, is
$(1+o(1))\frac{\log n}{\log\log n}$ (with an improved error term
compared to the results of \cite{ABLS}).

For which other families of $d$-regular graphs, with $d\geq 3$, is
this maximum $\Theta(\frac{\log n}{\log\log n})$?

\item The ``random setting'', where $n$ balls are thrown to $n$ bins, uniformly at random,
results in a maximal load of $(1+o(1))\frac{\log n}{\log\log n}$; it
seems plausible that this bound is the smallest maximal load
possible for a non-backtracking walk on any regular graph of degree
at least 3. Is it indeed true that for any $n$-vertex $d$-regular
graph with $d\geq 3$, a non-backtracking random walk of length $n$
visits some vertex at least $(1+o(1))\frac{\log n}{\log\log n}$ times almost
surely?
\end{itemize}

\noindent \textbf{Acknowledgement} We would like to thank Itai
Benjamini for useful discussions.

\end{document}